\documentclass[preprint,12pt]{elsarticle}
\usepackage{amssymb}
\usepackage{amsmath}
\usepackage{amsfonts}
\usepackage{latexsym}
\usepackage[T1]{fontenc}
\usepackage[latin1]{inputenc}
\usepackage[all,ps]{xy}
\RequirePackage{amsthm}
\RequirePackage{amssymb}
\usepackage{enumerate}
\usepackage{amscd}
\usepackage{verbatim}

\newtheorem{thm}{Theorem}

\newtheorem{rmk}{Remark}
\newtheorem{prop}{Proposition}
\newtheorem{lem}{Lemma}
\newtheorem{cor}{Corollary}
\newtheorem*{defn}{Definition}

\newcommand\sI{{\mathcal I}}

\newcommand\sO{{\mathcal O}}

\newcommand\bZ{{\mathbb Z}}

\newcommand\Hom{{\rm Hom}}

\newcommand\Ext{{\rm Ext}}

\newcommand\NS{{\rm NS}}

\journal{Journal of Geometry and Physics}

\begin{document}

\begin{frontmatter}

\title{Prioritary omalous bundles on Hirzebruch surfaces}
\author{Marian Aprodu}
\address{University of Bucharest, Faculty of Mathematics and Computer Science, 14 Academiei Str., 010014 Bucharest, Romania}
\address{Simion Stoilow Institute of Mathematics of the Romanian Academy, P.O. Box 1-764, 014700 Bucharest, Romania}
\author{Marius Marchitan}
\address{"\c Stefan cel Mare" University, Str. Universit\u a\c tii 13, 720229 Suceava, Romania}
\address{Integrated Center for Research, Development and Innovation in Advanced Materials, Nanotechnologies, and Distributed Systems for Fabrication and Control (MANSiD), "\c Stefan cel Mare" University, Suceava, Romania}
\address{Simion Stoilow Institute of Mathematics of the Romanian Academy, P.O. Box 1-764, 014700 Bucharest, Romania}


\begin{abstract}
An irreducible algebraic stack is called \emph{unirational} if there exists a surjective morphism, representable by algebraic spaces, from a rational variety to an open substack.
We prove unirationality of the stack of prioritary omalous bundles on Hirzebruch surfaces, which implies also the unirationality of the moduli space of omalous $H$-stable bundles for any ample line bundle $H$ on a Hirzebruch surface. To this end, we find an explicit description of the duals of omalous rank-two bundles with a vanishing condition in terms of monads. Since these bundles are prioritary, we conclude that the stack of prioritary omalous bundles on a Hirzebruch surface different from $\mathbb P^1\times \mathbb P^1$ is dominated by an irreducible section of a Segre variety, and this linear section is rational \cite{I}. In the case of the space quadric, the stack has been explicitly described by N. Buchdahl. 
As a main tool we use Buchdahl's Beilinson-type spectral sequence. Monad descriptions of omalous bundles on  hypersurfaces in $\mathbb P^4$, Calabi-Yau complete intersection, blowups of the projective plane and Segre varieties have been recently obtained by A. A. Henni and M. Jardim~\cite{HJ}, and monads on Hizebruch surfaces have been applied in a different context in~\cite{BBR}.
\end{abstract}


%
%

\end{frontmatter}

\section{Introduction}

An {\em omalous bundle} on a complex projective smooth variety is a vector bundle whose determinant is the anti-canonical bundle and with second Chern class equal to the class of the tangent bundle.
Omalous bundles have been introduced by Ron Donagi and the motivation comes from physics: these conditions on the Chern classes imply the usual Green-Schwartz anomaly cancellation conditions. The omality condition is used in the construction of quantum sheaf cohomology \cite{Gu}. In \cite{HJ} A. A. Henni and M. Jardim found explicit descriptions of (stable) omalous bundles on several types of varieties: hypersurfaces in $\mathbb P^4$, Calabi-Yau complete intersection, blowups of the projective plane and Segre varieties.

{\em Prioritary sheaves} on the projective plane were defined by A. Hirschowitz and Y. Laszlo in \cite{HL}. This notion was extented on birationally ruled surfaces by Ch. Walter in \cite{Wa93}, \cite{Wa95}. In \cite{Wa93} it is proved that if $H$ is a polarisation on a birationally ruled surface with a numerical condition (such polarisations always exist), then any $H$-semistable torsion-free sheaf is prioritary, hence this notion extends $H$-semistability. The advantage in working with prioritary sheaves instead of (semi)stable ones is that their definition is polarisation-free. The most important result on prioritary sheaves is the irreducibility and the smoothness of the stack \cite{HL}, \cite{Wa93}, \cite{Wa95}.

The aim of this paper is to prove unirationality of the stack of prioritary omalous bundles on Hirzebruch surfaces via an explicit monad description of prioritary omalous rank-two vector bundles with a vanishing condition on a Hirzebruch surface. For technical reasons, we work with duals of omalous bundles. This approach does not affect the result, as the dual of a prioritary rank-two bundle is also prioritary. 

We work over the field of complex numbers.
The outline of the paper is the following. In section \ref{sec:prel}, we set the notation and we recall some facts that will be used in the core of the paper. We discuss the numerical invariants associated to rank-two vector bundles on Hirzebruch surfaces and the canonical extensions \cite{BrS1}, \cite{BrS2}, Beilinson spectral sequences on Hirzeburch surfaces \cite{Bu87}, and the general theory of monads \cite{OSS}. In section \ref{sec:prioritary} we describe completely the duals of prioritary omalous bundles in terms of the associated numerical invariants, Proposition \ref{prop: prioritary}. A similar description is valid for bundles with arbitrary Chern classes, Remark \ref{rmk:prioritary}.
In section \ref{sec:substack} we find a necessary and sufficient condition, given by the vanishing of the space of sections of a suitable twist for a dual of an omalous bundle to be given by a specific monad, Theorem \ref{thm: monada}. This vanishing condition is satisfied by all the duals of (semi)stable omalous bundles and moreover  the duals of omalous bundles that satisfy these condition are necessary prioritary, Proposition \ref{prop:ireductibilitate}. Hence the stack of bundles with this condition sits between all the stacks of (semi)stable bundles and the stack of prioritary bundles, and the inclusions are strict, Remark \ref{rmk:inclusions}. The monad description obtained in Theorem \ref{thm: monada}  is used in Section \ref{sec:unirational} to prove that the stack of prioritary omalous bundles is dominated by a rational variety, which is a linear section of a Segre variety, Theorem \ref{thm: unirational}.  Hence the stack of prioritary omalous bundles is an irreducible, smooth, unirational stack of dimension $4$, Theorem \ref{thm: unirational}. 
 Consequently, all the moduli spaces of stable omalous bundles are unirational, Corollary \ref{cor: unirational}.

\medskip

{\em Acknowledgements.} The authors have been partly supported by the CNCS-UEFISCDI grant PN-II-ID-PCE-2012-4-0156. We are grateful to P. Ionescu for indicating us how to verify the rationality of the linear section of the Segre variety in Proposition \ref{prop: Z rational}. We also thank the referee for having made constructive comments that helped improving the presentation.

\section{Preliminaries }
\label{sec:prel} 

In this preliminary section we fix the notation and we recall some definitions and facts that will be used in the main sections \ref{sec:prioritary} and \ref{sec:substack}.


\subsection{Hirzebruch surfaces}
\label{subsec:conv}


Let $X=\Sigma_e$ be a Hirzebruch surface, $\Sigma_e=\mathbb P(\mathcal O_{\mathbb P^1}\oplus \mathcal O_{\mathbb P^1}(-e))\stackrel{\pi}{\rightarrow}\mathbb{P}^1$ with
$e\ge 0$.
Denote by $C_0=\mathcal{O}_X(1)$ the negative section ($C_0^2=-e$), and by $F$ a fibre of the ruling ($C_0\cdot F=1, F^2=0$).
Recall that $\mathrm{Pic}(X)=\mathbb{Z}\cdot C_0\oplus\mathbb{Z}\cdot F$ and the canonical bundle is $K_X=\mathcal O_X(-2C_0-(e+2)F)$. A line bundle $\mathcal O_X(aC_0+bF)$ has a nonzero global section if and only if $a\ge 0$ and $b\ge 0$. In what concerns the cohomology groups $H^1$ they are given by

\begin{lem}
\label{lem: h1}
Let $X=\Sigma_e$ be a Hirzebruch surface and $a,b\in\mathbb{Z}$. Then
\begin{equation*}
H^1(X,\mathcal{O}_X(aC_0+bF))\cong\left\{
	\begin{array}{rl}
H^0(\mathbb{P}^1, \bigoplus^{-a-1}_{k=1} \mathcal{O}_{\mathbb{P}^1}(ke+b)), & \text{if } a\le -2\\
		0, & \text{if } a = -1\\
H^0(\mathbb{P}^1, \bigoplus^{a}_{k=0} \mathcal{O}_{\mathbb{P}^1}(ke-b-2)), & \text{if } a\ge 0.
	\end{array}\right. 
\end{equation*}
\end{lem}

\subsection{Rank-two vector bundles on Hirzebruch surfaces as extensions}
\label{subsec:extinderi}

In this section we recall from \cite{BrS1} and \cite{BrS2} the numerical invariants naturally associated to rank-two bundles on Hirzebruch surfaces and the canonical extensions.

Let $V$ be a rank-two vector bundle on a Hirzebruch surface $X=\Sigma_e\stackrel{\pi}{\longrightarrow}\mathbb P^1$ with Chern classes $c_1(V)=\alpha C_0+\beta F$ and $c_2(V)=c_2\in \mathbb Z$. Since the fibres of the ruling are projective lines, we can speak about the generic splitting type of $V$ i.e.:
\[
V|_{F}\cong {\mathcal O}_F(d)\oplus {\mathcal O}_F(\alpha -d)
\]
for a general fibre $F$, where  $2d\ge\alpha$.  The integer $d$ is the first numerical invariant of $V$.

The second numerical invariant $r$ is obtained from a normalisation process:
\[
r=\mathrm{max}\lbrace \ell \in \mathbb Z|\ H^0(X,V(-dC_0-\ell F))\not=0\rbrace .
\]

In this context, we have the following result \cite{BrS1}, \cite{BrS2} (see also \cite{Fr} Chapter~6):


\begin{thm}
\label{thm:BrS}
Notation as above. There exists  $\zeta$ a zero-dimensional locally complete intersection subscheme of $X$ (or the empty set) of length $\ell(\zeta):=c_2 +\alpha (de-r)-\beta d+2dr-d^2e\geq 0$ such that $V$ is presented as an extension:
\begin{equation}
\label{eqn:canonical extension}
0\rightarrow {\mathcal O}_X(dC_0+rF)\rightarrow V\rightarrow {\mathcal O}_X((\alpha-d)C_0+
(\beta-r)F)\otimes\mathcal  I_\zeta\rightarrow 0.
\end{equation}
\end{thm}

The extension (\ref{eqn:canonical extension}) is called the {\em canonical extension} of $V$. This extension and the invariants $d$ and $r$ are very useful in a number of situations. For example, in \cite{ABCRAS2} a numerical stability criterion involving these invariants has been proved. The existence of vector bundles with given numerical invariants has been settled in \cite{ABNMJ}.

We will apply the canonical extension to prove a numerical criterion for bundles with certain Chern classes to be prioritary (this notion is recalled below) in section~\ref{sec:prioritary}.

\subsection{Prioritary sheaves}\label{subsec:prio}

The notion of prioritary sheaf was first introduced by A. Hirschowitz and Y. Laszlo in \cite{HL}, 
in order to include the class of stable bundles in a larger class. It consists of 
coherent sheaves $\mathcal E$ on $\mathbb{P}^2$, without torsion, which satisfy the property 
$\Ext^2 (\mathcal E,\mathcal E(-1))=0$, called {\em prioritary sheaves}, and  A. Hirschowitz and Y. Laszlo  proved that the latter form an irreducible space. Following the same path, Ch. Walter extended in \cite{Wa93}, \cite{Wa95} the notion of prioritary sheaf on birationally ruled surfaces. He considered 
$\pi:S\to C$ a birationally ruled surface and called prioritary a coherent sheaf $\mathcal E$ on $S$ 
which is torsion-free and $\Ext^2 (\mathcal E,\mathcal E(-f_p))=0$ for all $p\in C$, where $f_p=\pi^{-1}(p)$. 
He showed that, for given $r\ge 2, c_1\in \NS (S)$ and $c_2\in\bZ$, the stack 
${\mathrm{Prior}}_S(r,c_1,c_2)$ of prioritary torsion-free sheaves on $S$ of rank $r$ and Chern classes 
$c_1$ and $c_2$ is smooth and irreducible of dimension $-\chi(\mathcal E,\mathcal E)$, see \cite{Wa93} Proposition 2, \cite{Wa95} (1.0.1) and the proof of Proposition 3.1. Moreover, for any polarisation $H$ on $S$ with $H\cdot(K_S+f_p)<0$ for $p\in C$ any $H$-semistable sheaf is prioritary, \cite{Wa93}.

Note that the dual of a rank-two prioritary bundle $V$ on a birationally ruled surface remains prioritary. Indeed, we apply the formula $V^*=V\otimes \mathrm{det}(V)$ in the definition.

In the sequel, we will be concerned with prioritary rank-two bundles $V$ on a Hirzeburch surface $X$. In this case, the definition reduces to the condition 
\begin{equation}
\label{eqn: prioritary}
\Ext^2(V,V(-F))\cong H^2(X,V^*\otimes V(-F))=0
\end{equation}
where $F$ is the class of the fibre of the ruling. In section \ref{sec:prioritary} we will see that prioritary bundles with Chern classes $c_1=K_X$ and $c_2=4$ admit a precise numerical characterisation.

\subsection{Beilison spectral sequences}\label{subsec:Beilinson}

Beilinson spectral sequences have been defined first on a projective space by A. Beilinson with the aim of describing its derived category. Later on, similar constructions have been made on other classes of varieties (Grassmannians, hyperquadrics, Hirzebruch surfaces, scrolls etc). We recall here very briefly the case of a Hirzeburch surface
 $X=\Sigma_e$. If $\Delta\subset X\times X$ denotes the diagonal, N. Buchdahl \cite{Bu87} observed that $\Delta$ can be described scheme-theoretically as the zero-locus of a global section in a rank-two vector bundle over $X\times X$. This phenomenon produces Beilinson type spectral sequences on $X$.
Specifically, if $V$ is a vector bundle of arbitrary rank on $X$, there is a spectral sequence abutting to $V$ (see \cite{Bu87}):
\begin{equation}
E^{p,q}_1\Rightarrow
\left\{
\begin{array}{ll}
{V} & \mbox{if } p+q=0\\
0        & \mbox{otherwise.}
\end{array}
\right.
\label{spectral}
\end{equation}

Moreover, $E^{p,q}_1=0$ if $p\not\in\{-2,-1,0\}$ or if $q\not\in\{0,1,2\}$ and the remaining terms of the spectral sequence are described as follows, \cite{Bu87}:
\begin{equation}
\label{eqn: E_0}
E^{0,\,q}_1\cong H^q(X,V)\otimes\mathcal{O}_X,
\end{equation}
\begin{equation}
\label{eqn: E_2}
E^{-2,\,q}_1\cong H^q(X,V(-C_0-F))\otimes
\mathcal{O}_X(-C_0-(e+1)F)),
\end{equation}
and $E_1^{-1,q}$ can be computed from an exact sequence 
\begin{equation}
\label{eqn: E_1a}
H^q(X,V(-F))\otimes\mathcal{O}_X(-F)
\rightarrow E_1^{-1,q} \rightarrow
H^q(X,V(-C_0))\otimes\mathcal{O}_X(-C_0-eF).
\end{equation}

Using the Beilinson spectral sequence, we see that a vector bundle on $X$ is  determined by the cohomology of suitable twists and some vector bundle morphisms.

\subsection{Monads}\label{subsec:monad}

\begin{defn}
\label{defn:monads}
Let $X$ be a smooth projective variety. A {\em monad} on $X$ is a complex of vector bundles
\[
0\rightarrow A\stackrel{a}{\rightarrow}B\stackrel{b}{\rightarrow}C\rightarrow 0
\]
with $a$ injective and $b$ surjective. The cohomology at the middle of this complex is called the {\em cohomology of the monad}.
\end{defn}

In \cite{OSS}, it was noted that the cohomology of a monad is a vector bundle with precisely determined Chern classes. In the study of monads, a fundamental tool is represented by the following notion \cite{OSS}.

\begin{defn}
\label{defn:display} 
The {\em display} of a monad
\[
0\rightarrow A\stackrel{a}{\rightarrow}B\stackrel{b}{\rightarrow}C\rightarrow 0
\]
is the following commutative exact diagram:
\[
\xymatrix{
&&0\ar[d] & 0\ar[d] &\\
0\ar[r] & A\ar[r] \ar@{=}[d]&K\ar[r]\ar[d]&V\ar[r]\ar[d] &0\\
0\ar[r] & A\ar[r]^a &B\ar[r]\ar[d]^b&Q\ar[r]\ar[d] &0\\
&&C\ar[d] \ar@{=}[r] & C\ar[d] &\\
&&0&0&
}
\]
where $K=\mathrm{ker}(b)$ and $Q=\mathrm{coker}(a)$.
\end{defn}

Monads have been proved useful in describing various moduli spaces, see for example \cite{BH78}. A key fact is the following relation between morphisms of monads and vector bundle morphisms \cite{OSS}.

\begin{lem}
\label{lem:monads}
Let $E=H(M), E'=H(M')$ two vector bundles that represent the cohomology of two monads
\begin{equation*}
\begin{array}{lrl}
(M)&:\quad &0\rightarrow A\stackrel{a}{\rightarrow}B
    \stackrel{b}{\rightarrow}C\rightarrow 0\\
(M')&:\quad &0\rightarrow A'\stackrel{a'}{\rightarrow}B'
    \stackrel{b'}{\rightarrow}C'\rightarrow 0
\end{array}
\end{equation*}
over a smooth projective variety $X$. The map
$$
h:\Hom(M,M')\rightarrow\Hom(E,E')
$$
that sends every monad morphism to the corresponding bundle morphism is bijective if the following conditions are verified:
\begin{equation*}
\begin{array}{l}
  \Hom(B,A')=\Hom(C,B')=0,\\
  H^1(X,B^*\otimes A')=H^1(X,C^*\otimes B')=0,\\
  H^1(X,C^*\otimes A')=H^2(X,C^*\otimes A')=0.
\end{array}
\end{equation*}
\end{lem}

In section \ref{sec:substack} we obtain a precise characterisation using monads for a class of rank-two vector bundles on a Hirzebruch surface.

\section{A numerical criterion for prioritary bundles}
\label{sec:prioritary}

Using the canonical extensions (\ref{eqn:canonical extension}) from Theorem \ref{thm:BrS}, prioritary omalous bundles can be characterised in terms of invariants $d$ and $r$. For technical reasons, we shall study here, and in the next section, the {\em duals} of rank-two omalous bundles rather than these bundles themselves. The dual bundles have $c_1=K_X$ and $c_2=4$. From the point of view of prioritary bundles, it does not make any difference, as we have already noted that the dual of a prioritary bundle remains prioritary.

The precise description is contained in the following result:

\begin{prop}
\label{prop: prioritary}
Let $V$ be a rank-two vector bundle on a Hirzebruch surface $X$ with $c_1(V)=K_X$ and $c_2(V)=4$. Consider the canonical extension
\begin{equation}
\label{prio: 1}
0\to L_1\to V\to L_2\otimes \sI_\zeta\to 0,
\end{equation}
with $L_1=\sO_X(dC_0+rF)$, $L_2=\sO_X(-(d+2)C_0-(e+2+r)F)$ 
and $d\geq-1$. Then
\[
H^2(X,V^*\otimes V(-F))\cong H^0(X,\sO_X(2dC_0+(2r+1)F)).
\]
In particular, $V$ is prioritary if and only if $d=-1$ or $r\le -1$.
\end{prop}

\proof
From the isomorphism $V\cong V^*\otimes K_X$, and from Serre duality, we have to prove
\begin{equation}
\label{prio: 2}
H^0(X,V\otimes V(F))\cong H^0(X,\sO_X(2dC_0+(2r+1)F)).
\end{equation}

Twisting the sequence (\ref{prio: 1}) with $V(F)$ we obtain a sequence
\begin{equation}
\label{prio: 3}
0\to H^0(X,L_1\otimes V(F))\to H^0(X,V\otimes V(F))\to H^0(X,L_2\otimes \mathcal I_\zeta\otimes V(F))
\end{equation}

We claim that $H^0(X,L_2\otimes V(F))=0$. Indeed, twist the equence (\ref{prio: 1}) with $L_2(F)$ and obtain
\begin{equation}
\label{prio: 4}
0\to H^0(X,L_1\otimes L_2(F))\to H^0(X,L_2\otimes V(F))\to H^0(X,L_2^{\otimes 2}(F)\otimes \mathcal I_\zeta)
\end{equation}
Since $L_1\otimes L_2\cong K_X$ it follows that $H^0(X,L_1\otimes L_2(F))$ and, since $d\ge -1$, we have $H^0(X,L_2^{\otimes 2}(F))=0$, and the sequence (\ref{prio: 4}) implies the vanishing of $H^0(X,L_2\otimes V(F))=0$.
In particular, from (\ref{prio: 3}) we obtain an isomorphism $H^0(X,L_1\otimes V(F))\cong H^0(X,V\otimes V(F))$.

We compute $H^0(X,L_1\otimes V(F))$. From (\ref{prio: 1}) twisted by $L_1(F)$ we obtain the sequence
\[
0\to H^0(X,L_1^{\otimes 2}(F))\to H^0(X,L_1\otimes V(F))\to H^0(L_1\otimes L_2(F)\otimes \mathcal I_\zeta)
\]
Since $ H^0(L_1\otimes L_2(F))=0$, we obtain the isomorphism predicted in (\ref{prio: 2}).
\endproof

\begin{rmk}
\label{rmk:prioritary}
An identical proof works for vector bundles with arbitrary Chern classes. The precise statement that can be proved is that a rank-two vector bundle $V$ on $X$ with $c_1(V)=\alpha C_0+\beta F$  and associated numerical invariants $d$ and $r$ is prioritary if and only if 
\[
d=\left[\frac{\alpha +1}{2}\right]\mbox{ or }r<\frac{\beta+e+1}{2}.
\]
These conditions follow from the isomorphism 
\[
H^2(X,V^*\otimes V(-F))\cong H^0(X,\mathcal O_X((2d-\alpha)C_0+(2r-\beta-e-1)F)).
\]
\end{rmk}

%
%
%
%

\section{Monads for prioritary omalous bundles}
\label{sec:substack}


In this section, we will analyse rank-two vector bundles $V$ on $X$
with $c_1(V)=K_X$ and $c_2(V)=4$ which satisfy the extra-condition
\begin{equation}
\label{eqn:vanishing}
H^0(X,V(C_0+F))=0.
\end{equation}

Note that the cotangent bundle $\Omega^1_X$ satisfies the hypotheses. Indeed, from the natural extension
\begin{equation}
0\to\mathcal{O}_X(-2F)\to \Omega^1_X\to\mathcal{O}_X(-2C_0-eF)\to 0.
\end{equation}
we infer that $H^0(X,\Omega^1_X(C_0+F))=0$. In the same time, the sequence above destabilizes $\Omega^1_X$ with respect to any polarisation, hence $\Omega^1_X$ is never stable.

However, we can prove the following

\begin{prop}
Let $H$ be an ample line bundle on $X$ and $V$ be an $H$-stable rank-two vector bundle on $X$ with  $c_1(V)=K_X$ and $c_2(V)=4$. Then $V$ satisfies the condition (\ref{eqn:vanishing}).
\end{prop}

\proof
If $H^0(X,V(C_0+F))\ne 0$ then there is an injective map $\mathcal O_X(-C_0-F)\to V$. Since $H\cdot F>0$ and $c_1(V)=\mathcal O_X(-2C_0-(e+2)F)$ it follows that $\mu_H(\mathcal O_X(-C_0-F))\ge\mu_H(V)$ i.e. $\mathcal O_X(-C_0-F)$ is destabilising.
\endproof

\begin{rmk}
If $e>0$ then $H$-stability can be replaced by $H$-semistability in the statement above.
\end{rmk}

On the other hand, we have

\begin{prop}
\label{prop:ireductibilitate}
Any rank-two vector bundle $V$ on $X$ with $c_1(V)=K_X$, $c_2(V)=4$ and $H^0(X,V(C_0+F))=0$ is prioritary. In particular, these bundles form an irreducible stack $\mathcal N_X$, which is an open substack of the stack ${\mathrm{Prior}}_X(2,K_X,4)$ of prioritary bundles.
\end{prop}

\begin{proof}
We apply the numerical criterion of Proposition \ref{prop: prioritary}.
Twisting the canonical extension (\ref{prio: 1}) by $\mathcal{O}_X(C_0+F)$, 
$H^0(X,V(C_0+F))=0$ implies $H^0(X,\mathcal{O}_X((d+1)C_0+(r+1)F))=0$. 
Since $d+1\geq 0$ it follows that $r+1<0$. 

The irreducibility of the stack follows from \cite{Wa93} Proposition 2.
\end{proof}

\begin{rmk}
\label{rmk:inclusions}
The stack of bundles with $H^0(X,V(C_0+F))=0$ sits between the all the stacks of stable bundles and the stack of prioritary bundles. 
We mention that there are prioritary bundles that do not satisfy the vanishing (\ref{eqn:vanishing}). In fact, any vector bundle with $c_1(V)=K_X$, $c_2(V)=4$ and numerical invariants $d=-1$ and $r=-1$ is prioritary, from Proposition \ref{prop: prioritary} and has $H^0(X,V(C_0+F))\ne 0$. The existence of these bundles follows from \cite{ABNMJ} Theorem 10 (II).

The Riemann-Roch theorem implies immediately that the dimension of the stack $\mathcal N_X$ equals $4$. Indeed, for any $V$ omalous and prioritary bundle, we have $\chi(V,V)=\chi(X,V^*\otimes V)$, and $c_1(V^*\otimes V)=0$ and $c_2(V^*\otimes V)=4c_2(V)-c_1^2(V)=8$ and hence $\chi(V,V)=-4$. The stack is also smooth, \cite{Wa95}.
\end{rmk}

In the next result, we obtain a one-to-one correspondence between the bundles considered here and certain monads.

\begin{thm}
\label{thm: monada}
Let $V$ be a rank-two vector bundle on $X$
with $c_1(V)=K_X$ and $c_2(V)=4$. Then  $H^0(X,V(C_0+F))=0$ if and only if $V$ is the cohomology of a monad:
\begin{equation}
\label{eqn: monada}
0\rightarrow A\stackrel{a}{\rightarrow}B
\stackrel{b}{\rightarrow}C\rightarrow 0,
\end{equation}
where $A\cong\mathcal{O}_X(-C_0-(e+1)F)^{\oplus e}$, 
$B\cong\sO_X(-F)^{\oplus 2}\oplus\sO_X(-C_0-eF)^{\oplus(e+2)}$ and
$C\cong\mathcal{O}_X^{\oplus 2}$.
Moreover, in this description, isomorphism classes of monads map bijectively to isomorphism classes of bundles.
\end{thm}

\begin{proof}
We will repeatedly use the  isomorphism
$V\cong V^*\otimes K_X$. 

\medskip

{\em The "if" part.} Suppose that $V$ is given by the cohomology of a monad as in the statement. By duality, we have to prove that $H^2(X,V(-C_0-F))=0$. From the display of the monad, Definition \ref{defn:display},
twisted by $\mathcal O_X(-C_0-F)$ we obtain two short exact sequences
\begin{equation}
\label{eqn: display 1}
0\to A(-C_0-F)\to B(-C_0-F)\to Q(-C_0-F)\to 0
\end{equation}
and 
\begin{equation}
\label{eqn: display 2}
0\to V(-C_0-F)\to Q(-C_0-F)\to C(-C_0-F)\to 0.
\end{equation}

From the sequence (\ref{eqn: display 1}) we have $H^2(X,Q(-C_0-F))=0$. Since $C=\mathcal O_X^{\oplus 2}$ we also have $H^1(X,C(-C_0-F))=0$. Hence, from the sequence (\ref{eqn: display 2}) we infer that $H^2(X,V(-C_0-F))=0$.

\medskip

{\em The "only if" part}.
The hypothesis implies that $H^0(X,V)=0$, therefore 
$$
H^2(X,V)\cong H^0(X,V^*\otimes  K_X)\cong H^0(X,V)=0,
$$ 
and from (\ref{eqn: E_0}) it follows that $E_1^{0,0}=E_1^{0,2}=0$. 

The hypothesis also implies that $H^0(X,V(-C_0-F))=0$ and hence $E_1^{-2,0}=0$ from (\ref{eqn: E_2}).
On the other hand, from duality and the isomorphism $V\cong V^*\otimes  K_X$, it follows that
$H^2(X,V(-C_0-F))= H^0(X,V(C_0+F))=0$, hence $E_1^{-2,2}=0$, too.

To compute the terms $E_1^{-1,0}$ and $E_1^{-1,2}$ we use the exact sequence (cf. (\ref{eqn: E_1a}))
$$
H^0(X,V(-F))\otimes\mathcal{O}_X(-F)
\rightarrow E_1^{-1,0} \rightarrow
H^0(X,V(-C_0))\otimes\mathcal{O}_X(-C_0-eF),
$$
and, respectively,
$$
H^2(X,V(-F))\otimes\mathcal{O}_X(-F)
\rightarrow E_1^{-1,2} \rightarrow
H^2(X,V(-C_0))\otimes\mathcal{O}_X(-C_0-eF).
$$

As in the previous cases, we obtain
$$
H^0(X,V(-F))=H^0(X,V(-C_0))=H^2(X,V(-F))=H^2(X,V(-C_0))=0,
$$ 
i.e. $E_1^{-1,0}=E_1^{-1,2}=0$. 

We have proved that
$$
E_1^{p,0}=E_1^{p,2}=0,\forall p\in \{-2,-1,0\}
$$
which means that the first page of the Beilinson spectral sequence looks like in Figure \ref{E1}.
%
%
\begin{figure}[ht]
	\centering
$$\begin{array}{|c|c|c|}
  \hline
  0 & 0 & 0\\
  \hline
   E_1^{-2,\,1} & E_1^{-1,\,1} &  E_1^{0,\,1}\\
  \hline
  0 & 0 & 0\\
  \hline
\end{array}\, $$
	\caption{The first page of the spectral sequence}
	\label{E1}
\end{figure}

We compute next the remaining terms of $E_1$.

To compute $E_1^{0,1}\cong H^1(X,V)\otimes\sO_X$ we use the vanishing of $H^0(X,V)$ and of $H^2(X,V)$ and we apply the Riemann-Roch formula to obtain $h^1(X,V)=2$. It follows that
$E_1^{0,1}\cong \mathcal{O}_X^{\oplus 2}$.

To compute $E^{-2,1}_1\cong H^1(X,V(-C_0-F))\otimes\mathcal{O}_X(-C_0-(e+1)F)$ 
we apply again the Riemann-Roch formula: $\chi(X,V(-C_0-F))=-e$ and we use $H^0(V(-C_0-F))=H^2(V(-C_0-F))=0$. Hence
$E^{-2,1}_1\cong \mathcal{O}_X(-C_0-(e+1)F)^{\oplus e}$.

To compute $E^{-1,1}_1$ we take into account (\ref{eqn: E_1a}) 
and use the vanishing of 
$H^0(X,V(-C_0))$ and of $H^2(X,V(-F))$.
We deduce that
$E^{-1,1}_1$ lies in a short exact sequence
{\small
\begin{equation}
\label{eqn: E11}
0\to H^1(V(-F))\otimes\mathcal{O}_X(-F)
\rightarrow E_1^{-1,1} \rightarrow
H^1(V(-C_0))\otimes\mathcal{O}_X(-C_0-eF)\to 0.
\end{equation}
}
We compute $\chi(X,V(-F))=-2$ and $\chi(X,V(-C_0))=-e-2$. 
Since $H^0(X,V(-F))$, $H^2(X,V(-F))$, $H^0(X,V(-C_0))$ and $H^2(X,V(-C_0))$ are all equal to zero, 
we obtain $h^1(V(-F))=2$ and $h^1(V(-C_0))=e+2$. On the other hand,
\begin{align*}
&\Ext^1\big(H^1(X,V(-C_0))\otimes\sO_X(-C_0-eF),H^1(X,V(-F))\otimes\sO_X(-F)\big)=\\
&=H^1(X,V(-C_0))\otimes H^1(X,V(-F))\otimes \Ext^1(\sO_X(-C_0-eF),\sO_X(-F)).
\end{align*}

The isomorphism 
\[
\Ext^1(\sO_X(-C_0-eF),\sO_X(-F))\cong H^1(X,\sO_X(C_0+(e-1)F))
\] 
and Lemma \ref{lem: h1} for $a=1$ and $b=e-1$ yield to
$$
\Ext^1(\sO_X(-C_0-eF),\sO_X(-F))=0,
$$ 
which means that the sequence (\ref{eqn: E11}) splits and hence
$$
E_1^{-1,1}\cong \sO_X(-F)^{\oplus 2}\oplus\sO_X(-C_0-eF)^{\oplus(e+2)}.
$$

The differentials 
$d_1^{p,\,q}:E_1^{p,\,q}\rightarrow E_1^{p+1,\,q}$ 
define the complex
$$
E_1^{-2,\,1}\stackrel{a}{\rightarrow}E_1^{-1,\,1}
\stackrel{b}{\rightarrow}E_1^{0,\,1}.
$$
Let $K={\rm ker}\,a$, $L={\rm ker}\,b/{\rm Im}\,a$ 
and $M={\rm coker}\,b$. The second page $E_2$ of the spectral sequence is drawn in figure \ref{E2}.

%
%
\begin{figure}[ht]
	\centering
$$\begin{array}{|c|c|c|}
  \hline
  0 & 0 & 0\\
  \hline
   K & L & M\\
  \hline
  0 & 0 & 0\\
  \hline
\end{array}\, $$
	\caption{The second page of the spectral sequence}
	\label{E2}
\end{figure}

The differentials $d_2^{p,\,q}:E_2^{p,\,q}\rightarrow E_2^{p+2,\,q-1}$ all vanish
and then
$$
E_{\infty}^{-2,\,1}=K,\, E_{\infty}^{-1,\,1}=L,\, E_{\infty}^{0,\,1}=M.
$$
From (\ref{spectral}) we infer that
$$
K=M=0 \hspace{.5cm}\textrm{\c si} \hspace{.5cm} L=V,
$$
i.e.
$$
0\rightarrow E_1^{-2,\,1}\stackrel{a}{\rightarrow}
E_1^{-1,\,1}\stackrel{b}{\rightarrow}E_1^{0,\,1}\rightarrow 0
$$
is a monad, whose cohomology is $V$.

To prove that we have a bijection between isomorphism classes of monads and isomorphism classes of bundles, it suffices to show that (Lemma  \ref{lem:monads})
\begin{equation*}
\begin{array}{l}
  \Hom(B,A)=\Hom(C,B)=0,\\
  H^1(X,B^*\otimes A)=H^1(X,C^*\otimes B)=0,\\
  H^1(X,C^*\otimes A)=H^2(X,C^*\otimes A)=0.
\end{array}
\end{equation*}
These conditions are verified by direct computations and Serre duality.
\end{proof}

%

Monad descriptions of omalous bundles on  hypersurfaces in $\mathbb P^4$, Calabi-Yau complete intersection, blowups of the projective plane and Segre varieties have been recently obtained by A. A. Henni and M. Jardim~\cite{HJ}.

Note that the monad description does not automatically guarantee that it defines a morphism of stacks. This issue will be dealt with in the next section.

\section{The geometry of the stack of prioritary omalous bundles}
\label{sec:unirational}

In this section, we analyse the geometry of the stack of  prioritary omalous bundles, with emphasis on unirationality. An irreducible algebraic stack is called {\em unirational} if there exists a surjective morphism, representable by algebraic spaces, from a rational variety to an open substack of it.

Throughout this section, we assume that $e\ge 1$. Note that the case of $\mathbb P^1\times \mathbb P^1$ has been  completely described in \cite{Bu87} Proposition 1.

We begin with a study of the space $\mathcal Z$ parametrizing monads of type (\ref{eqn: monada}). With the notation from Theorem \ref{thm: monada}, the morphisms $a$ is given by two blocks of matrices
\[
a=\left(
\begin{array}{c}
a_1\\
a_2
\end{array}
\right)
\]
with
\[
a_1\in M_{2\times e}(H^0(X,\mathcal O_X(C_0+eF))),\ a_2\in M_{(e+2)\times e}(H^0(X,\mathcal O_X(F)))
\]
and the morphism $b$ is given by two blocks of matrices
\[
b=\left(
\begin{array}{cc}
b_1 & b_2
\end{array}
\right)
\]
\[
b_1\in M_{2\times 2}(H^0(X,\mathcal O_X(F))),\ b_2\in M_{2\times (e+2)}(H^0(X,\mathcal O_X(C_0+eF))).
\]

Denote by 
\[
M_1=M_{2\times e}(H^0(X,\mathcal O_X(C_0+eF)))\oplus M_{(e+2)\times e}(H^0(X,\mathcal O_X(F))),
\]
\[
M_2= M_{2\times 2}(H^0(X,\mathcal O_X(F)))\oplus M_{2\times (e+2)}(H^0(X,\mathcal O_X(C_0+eF))),
\]
\[
M_3= M_{2\times e} (H^0(X,\mathcal O_X(C_0+(e+1)F))).
\]

Recall that $h^0(\mathcal O_X(C_0+eF))=e+2$ and $h^0(\mathcal O_X(C_0+(e+1)F)=e+4$.  For $i\in\{1,2,3\}$ put $m_i=\mathrm{dim}(M_i)-1$ and compute $m_1=4e^2+8e-1$, $m_2=2e^2+8e+15$ and $m_3=2e^2+8e-1$.

We obtain the following description of $\mathcal Z$:
\[
\mathcal Z=\{(a,b)\in M_1\times M_2|\ a\mbox{ injective, }b\mbox{ surjective and }b_1\cdot a_1+b_2\cdot a_2=0\in M_3\},
\]
in particular, $\mathcal Z$ is a quasi-affine variety in $M_1\times M_2$ whose closure is given by the quadratic equations $b_1\cdot a_1+b_2\cdot a_2=0$. Using an idea from \cite{BBR}, we can prove even more:

\begin{prop}
\label{prop: Z smooth}
The variety $\mathcal Z$ is smooth of dimension $4(e^2+2e+4)$. 
\end{prop}

\proof
The proof is very similar to \cite{BBR}, Lemma 4.9. We only have to verify that $H^2(X,V^*\otimes V)=0$ which is implied by the condition (\ref{eqn: prioritary}). This condition implies that the map 
\[
\mu:M_1\times M_2\to M_3,\ (a,b)\mapsto b_1\cdot a_1+b_2\cdot a_2
\]
is smooth and surjective, see \cite{BBR}. In particular, the dimension of $\mathcal Z=\mu^{-1}(0)$ equals $\mathrm{dim}(M_1)+\mathrm{dim}(M_2)-\mathrm{dim}(M_3)$.
\endproof

\begin{prop}
\label{prop: Z rational}
The variety $\mathcal Z$ is rational.
\end{prop}

\proof
Consider the projection $M_1\times M_2\to \mathbb PM_1\times \mathbb PM_2$ and note that the fibres $\mathbb C^*\times \mathbb C^*$ of this projection over the image $\mathcal Y$ of $\mathcal Z$ are completely contained in $\mathcal Z$. Hence it suffices to prove that $\mathcal Y$ is rational.

Embed 
$
\mathbb PM_1\times \mathbb P M_2
$
by Segre in $\mathbb P(M_1\otimes M_2)$. The linear map 
\[
M_1\otimes M_2\to M_3,\ (a_1, a_2)\otimes (b_1, b_2)\mapsto b_1\cdot a_1+b_2\cdot a_2
\]
induces a rational linear projection $\mathbb P(M_1\otimes M_2)\dashrightarrow \mathbb PM_3$. If $K$ denotes the kernel of the linear map above, then the center of the projection is $\mathbb PK$. The intersection of $\mathbb PK$ with the Segre variety $\mathbb PM_1\times \mathbb PM_2$ is isomorphic to $\mathcal Y$.

We prove the rationality of this linear section of the Segre variety using an argument indicated by P. Ionescu \cite{I}. 
Recall that for $i\in\{1,2,3\}$ we defined $m_i=\mathrm{dim}(\mathbb PM_i)$ and computed $m_1=4e^2+8e-1$, $m_2=2e^2+8e+15$ and $m_3=2e^2+8e-1$. Put $m=\mathrm{dim}(\mathbb PK)$. Since $m+1\ge (m_1+1)(m_2+1)-(m_3+1)$ we obtain $m\ge m_1m_2+m_1+m_2-m_3-1$ and hence $m+m_2=m+m_3+ 16\ge m_1m_2+m_1+m_2$. In particular, $\mathbb PK$ intersects each fibre of the projection $\mathbb PM_1\times \mathbb PM_2\to \mathbb PM_1$ and hence $\mathbb PK\cap(\mathbb PM_1\times \mathbb PM_2)$ dominates $ \mathbb PM_1$. Since the intersection of $\mathbb PK$ with these fibres are projective subspaces, we infer that $\mathbb PK\cap(\mathbb PM_1\times \mathbb PM_2)$ is birational to a projective bundle over $\mathbb PM_1$ and hence is rational.
\endproof

\begin{prop}
\label{prop: nat morph}
\label{prop: surject}
There is a natural morphism $p:\mathcal Z\to \mathcal N_X$ which is representable by algebraic spaces and surjective. In particular, $p$ is a smooth atlas for $\mathcal N_X$.
\end{prop}

\proof
It suffices to give a natural rank-two vector bundle $\mathcal V$ on $X\times \mathcal Z$ with  $c_1(\mathcal V_z)=K_X$, $c_2(\mathcal V_z)=4$ and  $H^0(X,\mathcal V_z(C_0+F))=0$ for all $z\in \mathcal Z$. Then, if $S$ be a scheme over $\mathbb C$, let $f:S\to \mathcal Z$ be a morphism, the morphism $p$ associates to $f$ the vector bundle $(\mathrm{id}_X\times f)^*(\mathcal V)$ on $X\times S$.

We have natural vector bundle morphisms $\alpha:p_1^*A\to p_1^*B$, and $\beta:p_1^*B\to p_1^*C$ defined fiberwise as follows. Any point $z\in \mathcal Z$ corresponds to a monad, and hence to a pair $(a_z,b_z)$ of morphisms $a_z:A\to B$ and $b_z:B\to C$. Over any pair $(x,z)\in X\times\mathcal Z$, we defined $\alpha_{(x,z)}=a_z$ and $\beta_{(x,z)}=b_z$. It is clear that $\alpha$ is injective, $\beta$ is surjective, and $\beta\circ\alpha=0$ and hence we have a monad on $X\times \mathcal Z$, $0\to p_1^*A\to p_1^*B\to p_1^*C\to 0$ whose cohomology $\mathcal V$ is the vector bundle we were looking for.

Representability follows from the fact that the stack of vector bundles is an Artin stack, and hence the diagonal of $\mathcal N_X$ is representable which implies that any morphism from a scheme to $\mathcal N_X$ is representable, \cite{Vi}, Proposition~7.13.

Since the morphism $p$ is representable by algebraic spaces, surjectivity follows from surjectivity on objects, \cite{Stack} Lemma 77.7.3, and surjectivity on objects in ensured by Theorem \ref{thm: monada}.
\endproof

\begin{rmk}
\label{rmk: quot}
Denote by $G$ the group 
\[
G:=\mathrm{Aut}(A)\times \mathrm{Aut}(B)\times \mathrm{Aut}(C).
\] 

Obviously, we have an action of $G$ on $\mathcal Z$ given by 
\[
(\alpha,\beta,\gamma)\cdot(a,b):=(\beta a\alpha^{-1},\gamma b\beta^{-1}),
\]
and hence we have an induced quotient stack $\left[\mathcal Z/G\right]$ of dimension 
\[
\mathrm{dim}(\left[\mathcal Z/G\right])=\mathrm{dim}(\mathcal Z)-\mathrm{dim}(G)=2e^2+4e+4.
\]

Unlike the situation of \cite{BBR}, the action of $G$ is not free, and hence the quotient stack cannot be a variety. 
One can prove that $\mathcal Z\to\mathcal N_X$ factors through a natural morphism of stacks $[\mathcal Z/G]\to \mathcal N_X$, and, from Proposition \ref{prop: surject}, it is representable by algebraic spaces and surjective. However, by dimension reasons (recall that $\mathcal N_X$ is of dimension $4$), the two stacks are not isomorphic. 
\end{rmk}

Propositions \ref{prop: Z rational} and  \ref{prop: surject}  provide the necessary ingredients to show the main result of our paper:

\begin{thm}
\label{thm: unirational}
Suppose that $e\ge 1$. The stack $\mathcal N_X$ of rank-two prioritary bundles with Chern classes $c_1=K_X$ and $c_2=4$ is unirational.
\end{thm}


Since stable bundles are prioritary, Theorem \ref{thm: unirational} implies the following:

\begin{cor}
\label{cor: unirational}
For any ample line bundle $H$ on a Hirzebruch surface $X$, the moduli space $\mathcal M_H(K_X,4)$ of omalous $H$-stable bundles is unirational.
\end{cor}

\begin{rmk}
From \cite{ABCRAS2}, it follows that if $H$ belongs to a chamber of type $(K_X,4)$, non-adjacent to the class of a fibre, then the moduli space $\mathcal M_H(K_X,4)$ is non-empty. 
\end{rmk}

\bibliographystyle{elsarticle-num}

\end{document}